\documentclass[12pt]{amsart}
\usepackage{a4wide}
\usepackage{amscd}
\usepackage{amssymb}
\usepackage{amsthm}
\usepackage{amsmath}
\usepackage{latexsym}
\usepackage[hidelinks]{hyperref}
\usepackage[pdftex]{graphicx}
\usepackage{subfigure}
\usepackage{dsfont}
\usepackage{color}
\usepackage{algorithm}
\usepackage{algorithmic}

\newtheorem{theorem}{Theorem}
\theoremstyle{plain}

\newtheorem{lemma}[theorem]{Lemma}

\newtheorem{remark}[theorem]{Remark}

\numberwithin{equation}{section} \numberwithin{theorem}{section}
\newcommand{\R}{\ensuremath{\mathbb{R}}}
\newcommand{\E}{\ensuremath{\mathbb{E}}}
\newcommand{\Var}{\ensuremath{\mathbb{V}}}
\newcommand{\M}{\ensuremath{\mathbb{M}}}
\newcommand{\var}{\Var}

\newcommand{\Ord}{{\mathcal{O}} }

\newcommand{\N}{{\mathbb{N}}}

\newcommand{\Comment}[1]{}

\newcommand{\eps} {\varepsilon}
\renewcommand{\i}{\ifmmode\mathit{\mathchar"7010 }\else\char"10 \fi}
\renewcommand{\j}{\ifmmode\mathit{\mathchar"7011 }\else\char"11 \fi}

\newcommand{\Dx}{\Delta x}

\newcommand{\Dt}{\Delta t^n}

\newcommand{\cF}{\mathcal{F}}

\newcommand{\cN}{\mathcal{N}}

\newcommand{\rR}{\mathds{R}}

\def\e{{\mathrm{e}}}
\allowdisplaybreaks

\begin{document}

\title[Uncertainty quantification for hyperbolic problems]{Uncertainty quantification for linear hyperbolic equations with stochastic process or random field coefficients.}

\author[A. Barth]{Andrea Barth}
  \address[Andrea Barth]{\newline SimTech, University of Stuttgart\newline Pfaffenwaldring 5a\newline70569 Stuttgart}
  \email[]{andrea.barth\@@mathematik.uni-stuttgart.de}

\author[F. G. Fuchs]{Franz G. Fuchs} \address[Franz Georg Fuchs]{\newline Sintef ICT
  \newline Forskningsveien 1\newline
  N--0314 Oslo, Norway} \email[]{franzgeorgfuchs@gmail.com}


\begin{abstract}
    In this paper hyperbolic partial differential equations with random coefficients are discussed. Such random partial differential equations appear for instance in traffic flow problems as well as in many physical processes in random media.
    Two types of models are presented: The first has a time-dependent coefficient modeled by the Ornstein--Uhlenbeck process. The second has a random field coefficient with a given covariance in space.
    For the former a formula for the exact solution in terms of moments is derived.
    In both cases stable numerical schemes are introduced to solve these random partial differential equations. Simulation results including convergence studies conclude the theoretical findings.
\end{abstract}

\maketitle


\section{Introduction}

Hyperbolic partial differential equations with random data have been an active research field over the last decades. In ample situations measurements are not accurate enough to allow an exact description of a physical phenomena by a deterministic model.  
%
%
Uncertainty may then be introduced in the appropriate parameters and the distribution of the (now stochastic) solutions is studied. 
As en example, hyperbolic partial differential equations with random coefficients are applied in the modeling of underground water flow in porous media or, more general, of transport processes in non-uniform media, in the modeling of pollution spread and heat transfer and in traffic simulations. 
Those types of phenomena can be modeled by hyperbolic conservation laws that have the general from
\begin{equation}\label{eq:general_hyp}
    u_t + f(x,t,u)_x = 0
\end{equation}
in one spatial dimension, i.e., $x\in D\subset \rR$.
As mentioned, in many \emph{realistic} applications it is often the case that there are uncertainties in the parameters of the function $f$, or that uncertainty is even intrinsic to the problem. One way to model this is the following.
Given a probability space $(\Omega, \mathcal{F}, P)$ we can incorporate those uncertainties by considering the equation
\begin{align}\label{eq:general_stochhyp}
    \begin{split}
        u_t(x,t,\omega) + f(x,t,\omega,u(x,t,\omega))_x &= 0,\\
        u(x,0,\omega) &= g(x),
    \end{split}
\end{align}
where $f$ is a (in general nonlinear) function that now depends not only on space, time, and the unknown function $u$, but also on a stochastic variable $\omega \in \Omega$ that accounts for the uncertainties in the parameters of the conservation law.
A random function $u:D\times[0,T]\times\Omega\rightarrow\R$ for which Equation~\eqref{eq:general_stochhyp} holds P-almost everywhere in $\Omega$ (that is almost surely) is called a (strong) solution.
We are then interested in the distribution or in the evolution of certain moments of the solution of this equation, typically of the expectation $\E(u)$ and the variance $\var(u)$.



We restrict our attention to linear advection equations with a random transport velocity as a prototype problem.
There are many results in literature for hyperoblic equations with coefficients that are (real-valued) random variables, i.e. which do not depend on space or time. For instance, the authors in ~\cite{OL, DC08, CD07} present both theoretical results and numerical approximations.
In~\cite{DC11} the authors present expressions for the distribution of the solution of a linear advection equation with a time-dependent velocity, given in terms of the probability density function of the underlying integral of the stochastic process. Concrete results are presented in the case where the velocity field is deterministic, a random variable and Gaussian. 
Further, the same authors introduce numerical schemes for the mean of the solution of the linear transport equation with homogeneous random velocity and random initial conditions in~\cite{DC07} and the authors in~\cite{DFC09} extend the setting to Gaussian processes and telegraph processes.
In~\cite{JSK02} the linear advection equation with space- and time-dependent coefficients are subject of research. The authors develop numerical methods using polynomial chaos to solve the advection equation with a transport velocity given by a Gaussian or a log-normal distribution. In~\cite{BF16} we applied similar methods, like the ones developed here, to the magnetic induction equation and linear acoustics, both with a time- and space-dependent random background velocity field.

In order to approximate the moments of equations of type~\eqref{eq:general_stochhyp} numerically, methods are either based on a Monte Carlo approach or use a stochastic Galerkin or, more general, a polynomial chaos approach (see~\cite{JSK02, GX08, TZ10} and references therein). The latter approach is not suitable for any distribution. So far this approach is limited to uniform or Gaussian distributed fields or processes. A Monte Carlo method, on the other hand may also be used when dealing with jump processes or L\'evy random fields. This comes, however, to the price of a lower convergence rate of the Monte Carlo method.
We point out that a more efficient multilevel Monte Carlo approximation could be used in this article, but we refrain from doing so, since we wish to focus on the numerical approximation in the temporal and spatial domain as well as the approximation of the coefficient. For a result on the convergence and computational complexity of the multilevel Monte Carlo approximation for general Hilbert-space-valued random variables we refer to~\cite{BL12}. For a multilevel Monte Carlo finite volume method see for instance \cite{MSS}.
A further advantage of a Monte Carlo method based approximation is that it is non-intrusive, meaning that already implemented numerical solvers can be readily used. In addition, it does not depend on the correlation length of the stochastic input, leading to a large number of Karhunen--Lo\`eve terms for weakly correlated fields.

The article is structured as follows. In the first section we examine the linear transport equation with a time dependent coefficient $a = (a(t), t\in[0,T])$ given by the Ornstein--Uhlenbeck process. We derive a closed form expression for the moments of the distribution of the solution. We thereby extend the result found in~\cite{OL} and~\cite{DC11}. 
Furthermore, we introduce a second order (in space and time) Monte Carlo method to approximate the solution. We present simulation results and a convergence study.
The last section presents the linear transport equation with a space-dependent coefficient $a= (a(x), x\in D)$, assumed to be a Gaussian/L\'evy random field over the domain $D$. Here, we also present a second order (in space and time) Monte Carlo method for the approximation of the solutions. We show simulations and a self-convergence study.
Although, in both cases the random transport equation is scalar and linear, we see interesting effects in the moments of the solution that differ from the deterministic variants. Furthermore, the numerical methods/discretizations for the approximation of moments of the solution to the equations become non-trivial.

\section{Time-dependent uncertainty modeled by the Ornstein--Uhlenbeck process.}
In this section we are concerned with the distribution of the solution to the random partial differential equation
\begin{align}\label{eq:problemTime}
    \begin{split}
        u_t(x,t,\omega) + \left(a(t,\omega)\,u(x,t,\omega)\right)_x &= 0,\\
        u(x,0,\omega) &= g(x)
    \end{split}
\end{align}
where we model uncertainty in a way that allows for changes over time. That means we want to solve an advection equation with a time-dependent stochastic advection parameter.
Let us start by defining $a = (a(t),t\in[0,T])$ as the solution of the Ornstein--Uhlenbeck process
\begin{align}\label{eq:OU}
    \begin{split}
        d a(t) &= \theta(\mu - a(t)) dt + \sigma dW(t),\\
        a(0) &= a_0,
    \end{split}
\end{align}
where $W=(W(t), t\in[0,T])$ is a standard Brownian motion and $\mu\in\R$, $\theta>0$ and $\sigma>0$ are parameters. In general the initial condition can be random as well.
A standard Brownian motion or Wiener process, defined on the probability space $(\Omega, \cF,P)$, is a continuous stochastic process which starts in zero $P$-a.s and has independent and normally distributed increments, i.e., $W_t-W_s\sim\cN(0,t-s)$.
The idea of equation~\eqref{eq:OU} is that there are two competing features, one is the introduction of noise via the process $W$, the other is the relaxation of the solution to the mean (see Figure~\ref{fig:problemTimeExact}(c)) for some sample solutions).
For every $t\in[0,T]$ the random variable $a(t)$ is normally distributed with mean and variance
\begin{align}\label{eq:OU_moments}
    \begin{split}
        \E(a(t)) &= \mu + (a_0-\mu)e^{-\theta t},\\
        \Var(a(t)) &= \frac{\sigma^2}{2\theta}(1-e^{-2\theta t}).
    \end{split}
\end{align}
\begin{remark}
    Mean and Variance of teh Ornstein--Uhlenbeck process can be easily calculated by using It\^o's formula with the function $f(t,x) = \e^{\theta t} x$, and looking at the dynamics of $f(t,a(t))$. This leads to the following solution of the Ornstein--Uhlenbeck process
    \begin{equation}
        \label{eq:OU_direct}
            a(t) = \mu + \e^{-\theta t}(a_0-\mu) + \sigma \int_0^t \e^{-\theta (t-s)}\, dW(s).
    \end{equation}
    From this form we can directly deduce the expectation of $a(t)$
    and the variance is derived by using It\^o's isometry. 
\end{remark}

\subsection{Theoretical results.}
In the specific case of a time-dependent coefficient we can calculate a closed form of the distribution of the solution. For the moments of the solution of equation~\eqref{eq:problemTime} we have the following result (see Figure~\ref{fig:problemTimeExact} for an example).
\begin{theorem}\label{theorem_lineartime}
    The moments of the solution to Equation~\eqref{eq:problemTime} with coefficient $a$ given by the Ornstein--Uhlenbeck process~\eqref{eq:OU} exist and are given by
    \begin{equation}\label{eq:time_coeff_moment0}
            \E(u(x,t)) = \int g(x-y)f_{A(\hat{\sigma}^2,\hat{\mu})}(y) \,dy = (f_{A(\hat{\sigma}^2,0)}*g)(x-t\hat{\mu}),
    \end{equation}
    where $g(x) = u(x,0)$ and the probability density function $f_A$ is given by
    \begin{equation*}
        f_{A(\hat{\sigma}^2,\hat{\mu})}(y) = \frac{1}{\sqrt{2\pi \hat{\sigma}^2}} e^{-\frac{(y-\hat{\mu})^2}{2\hat{\sigma}^2}},
    \end{equation*}
    with diffusion coefficient $\hat{\sigma}^2=\frac{\sigma^2}{\theta^3} \big(\theta t + 2\e^{-\theta t} - \frac{1}{2}\e^{-2\theta t} - \frac{3}{2}\big)$ and transportation speed $\hat{\mu}=\mu - (a_0-\mu)\frac{\e^{-\theta t}-1}{\theta t}$. 
    As usual, $f*g$ denotes the convolution of the two functions.
\end{theorem}
\begin{remark}
    Higher moments of the solution can be calculated by
    \begin{align}
            \M_m(u(x,t)) = \E\big((u(x,t)-\E(u(x,t)))^m\big).
    \end{align}
\end{remark}

\begin{proof}
    The solution for a single realization (for a fix $\omega\in\Omega$) of Equation~\eqref{eq:problemTime} is given by $g(x-\int_0^ta(s,\omega)\,ds)$. 
    We start by calculating the first moment of this expression, i.e.
    \begin{equation*}
        \E(g(x-\int_0^ta(s)\,ds)).
    \end{equation*}
    That means, we have to calculate the distribution of the time integral over $a$, i.e. the distribution of the stochastic process
    \begin{equation*}
        A(t) = \int_0^t a(t)\, dt.
    \end{equation*}
    The process $A$ is again a Gaussian process, i.e. $A(t)\sim\cN(\hat{\mu}, \hat\sigma^2)$, and therefore completely characterized by its mean and variance. Using Fubini's theorem we have that
    \begin{align}
        \label{eq:OUtime_mu}
        \begin{split}
            \E(A(t)) 
            = \int_0^t \E(a(s))\, ds
            = \int_0^t \mu + \e^{-\theta s}(a_0-\mu)\, ds
            = \mu t - (a_0-\mu)\frac{\e^{-\theta t}-1}{\theta}
            =:\hat{\mu}.
        \end{split}
    \end{align}
    We express the variance of $A$ via the covariance of $A$ with itself
    \begin{equation*}
        \var(A(t)) = \text{Cov}(A(t),A(t)) = \E\big((A(t)-\E(A(t)))(A(t)-\E(A(t)))\big).
    \end{equation*}
    Using $A(t)-\E(A(t))=\sigma\int_0^t\int_0^s\e^{-\theta(s-u)}\,dW(u)\,ds$ (combine Equations~\eqref{eq:OU_direct} and~\eqref{eq:OUtime_mu}) this yields
    \begin{align*}
        \var(A(t)) &= \E\big(\sigma \int_0^t \int_0^s \e^{-\theta(s-u)}\,dW(u)\,ds\, \sigma \int_0^t \int_0^r \e^{-\theta(r-v)}\,dW(v)\,dr\big)\\
        &= 2\sigma^2 \int_0^t \e^{-\theta s}\int_0^t\e^{-\theta r} \,\E\big( \int_0^s \e^{\theta u}\,dW(u) \int_0^r \e^{\theta v}\,dW(v)\big)\,dr\,ds,
    \end{align*}
    using Fubini's theorem.
    For a Brownian motion $W$, it is known that $$\E( \int_0^s \e^{\theta u}\,dW(u) \int_0^r \e^{\theta v}\,dW(v))= \frac1{2\theta}(e^{2\theta \text{min}(s,r)}-1).$$
    Therefore, we have
    \begin{align*}
        \var(A(t)) &= 2\sigma^2 \int_0^t \e^{-\theta s}\int_0^s\e^{-\theta r} \frac1{2\theta}(e^{2\theta \text{min}(s,r)}-1)\,dr\,ds\\
        &= \frac{\sigma^2}{\theta} \int_0^t \e^{-\theta s}\int_0^s\e^{-\theta r}(e^{2\theta r}-1)\,dr\,ds
        = \frac{\sigma^2}{\theta^3} \big(\theta t + 2\e^{-\theta t} - \frac{1}{2}\e^{-2\theta t} - \frac{3}{2}\big)
        =:\hat{\sigma}^2.
    \end{align*}
    This gives us the variance of $A(t)$ depending on the variables $\theta$ and $\sigma$.

    Therefore, the expectation of the solution of equation~\eqref{eq:problemTime} is given by
    \begin{align*}
        \E(g(x-\int_0^ta(s)\,ds)) = \E(g(x-A(t)))\\
        =\int_{-\infty}^\infty g(x-y) f_A(y) \,dy
    \end{align*}
    where $f_A$ is the normal density function with parameters $\hat{\mu}$ and $\hat{\sigma}^2$ given by
    \begin{equation*}
        f_A(y) = \frac1{\sqrt{2\pi\hat{\sigma}^2}}\e^{-\frac{(y-\hat{\mu})^2}{2\hat{\sigma}^2}}.
    \end{equation*}
\end{proof}
\begin{remark}
    For the limit $\theta \rightarrow 0$, we recover the corresponding result for a pure Brownian motion process (i.e. $a(t)= \sigma W(t)$), where $\hat{\mu} = \mu$ and $\hat{\sigma}^2 = \sigma^2\frac{t^3}{3}$.
    This can be shown by a Taylor expansion as
    \begin{align*}
        \var(A(t)) &= \frac{\sigma^2}{\theta^3} \big(\theta t + 2\e^{-\theta t} - \frac{1}{2}\e^{-2\theta t} - \frac{3}{2}\big)\\
                   &= \frac{\sigma^2}{\theta^2} \Big(\theta t + 2\left(1 -\theta t +\theta^2 t^2/2 -\theta^3 t^3/3! + \Ord(\theta^4)\right)\\
                   &\quad \quad \quad - \frac{1}{2}\left(1 -2\theta t +4\theta^2 t^2/2 - 2^3\theta^3 t^3/3! + \Ord(\theta^4)\right)- \frac{3}{2}\Big)\\
                   &= \sigma^2 t^3/3 + \Ord(\theta).
    \end{align*}
    A similar Taylor expansion shows the result for $\E(A(t))$.
\end{remark}

Although we have a formula for the moments of the solutions to the linear advection equation with a velocity field given by the Ornstein-Uhlenbeck proccess, it will not be possible to obtain analytical solutions for a general hyperbolic equation and/or a general stochastic process. As this is a prototype problem, we therefore introduce a Monte Carlo based approximation of the solutions to Equation~\eqref{eq:problemTime} in the following.

\subsection{First order discretizations.}\label{sec:firstorderTime}
For the approximation of the (moments of the) solution to partial differential equations with random coefficient we have to discretize in space and time, as well as in the ``stochastic domain".
Here we use a Monte Carlo method with underlying first and higher order schemes (in space and time).
That means, that for each realization $\omega$ of Equation~\eqref{eq:problemTime} we have to approximate the (deterministic) solution of a hyperbolic partial differential equation. Our base method for each realization is, therefore, a finite volume scheme, see e.g.~\cite{LEV1} and references therein.

Before we continue with a technical description of the schemes used, we introduce some useful notation. As usual, $\Dx$ denotes the equidistant spatial step size. For $i=1,...,I$,  $I\in\N$, the cell centers are given by $x_i=(i-\frac{1}{2})\Dx$ together with the according cell interfaces $x_{i-1/2}=(i-1)\Dx$ for $i=1,...,I+1$. Similarly, $\Dt$ is the varying temporal step size leading to the discrete times $t^n= \sum_{i=1}^n \Delta t^i$ for $n\in\N$. For a function $b(x,t)$, we set $b_i^n=b\left(x_i, t^n\right)$.


A finite volume scheme is obtained by integrating Equation~\eqref{eq:problemTime} over some time interval $T^n=[t^{n-1},t^n]$, $t^n=t^{n-1} + \Dt$, (where $\Dt$ is still to be determined) and a control volume $X_i=[x_{i-1/2},x_{i+1/2}]$, leading to
\begin{equation*}
    0=\int_{T^n}\frac{1}{|X_i|}\int_{X_i} u_t + \left(a(t) u\right)_x \,dx \,dt
    =\frac{1}{|X_i|}\int_{X_i} u \,dx|_{t=t^{n-1}}^{t^n} + \frac{1}{|X_i|}\int_{T^n}\left(a(t) u\right) \,dt|_{x=x_{i-1/2}}^{x_{i+1/2}}.
\end{equation*}
Denoting the cell averages by $u_i(t) = \frac{1}{|X_i|}\int_{X_i} u \,dx$, we may write
\begin{equation}
    \label{eq:FVtime}
    u_i(t^n)= u_i(t^{n-1}) - (F^n_{i+1/2} - F^n_{i-1/2})/|X_i|,
\end{equation}
where the flux $F^n_{i+1/2}$ approximates the following integral 
\begin{equation*}
    F^n_{i+1/2} \approx \int_{T^n}\left(a(t) u\right) \,dt|_{x=x_{i+1/2}} = \int_{T^n}\left(a(t) \right) \,dt \,u|_{x=x_{i+1/2}}.
\end{equation*}
One possibility for this approximation is the standard upwind stencil, see~\cite{LEV1}
\begin{equation}
    \label{eq:Fluxtime}
    F^n_{i+1/2} - F^n_{i-1/2} = \text{max}(a^n,0)(u_i^{n-1}-u_{i-1}^{n-1}) + \text{min}(a^n,0)(u_{i+1}^{n-1}-u_i^{n-1}),
\end{equation}
where $a^n \approx \int_{T^n}a(t)\,dt$. We approximate this integral by choosing a point $t^* \in[t^{n-1},t^n]$, usually $t^*=t^{n-1}$, and setting
\begin{equation}
    \label{eq:timeintegral}
    a^n = \Dt a(t^*).
\end{equation}

In order to obtain the values $a^n$ we have to approximate the Ornstein--Uhlenbeck process, that is we need a discretization of the solution to Equation~\eqref{eq:OU}.
We use an implicit Euler--Maruyama method (which is in this case equal to the Milstein method, since $\sigma$ is a constant) for the potentially stiff ODE
$$a(s^{l+1}) - a(s^l) = \Delta s\theta\left(\mu - a(s^{l+1})\right) + \sigma \sqrt{\Delta s} Y^l, \ l\in\N$$.
This is equivalent to
\begin{equation}
    \label{eq:OUAppr}
    \begin{aligned}
        a^0 &= \mu, \\
        a^{l+1} &= \frac{a^l + \Delta s\theta\mu + \sigma\sqrt{\Delta s} Y^l}{1 + \Delta s\theta}
    \end{aligned}
\end{equation}
where $(Y^l, l\in\N)$ is a sequence of independent $\mathcal{N}(0,1)$-distributed random variables.
For a good approximation of the Ornstein--Uhlenbeck process used in the Monte Carlo simulation of Equation~\eqref{eq:problemTime}, i.e. in Equation~\eqref{eq:FVtime}, the constant step size $\Delta s$ should be chosen small enough, such that $\Delta s \leq \Delta t^n$, for all $n$, at least roughly.
In the simulations we choose
\begin{equation}
    \label{eq:OUApprDs}
    \Delta s = \frac{T}{\lceil3 T \lambda / \Delta x\rceil}, \mbox{ where } \lambda=\mu+\sigma.
\end{equation}
We would like to summarize the above steps in the following algorithm.
\begin{algorithm}[H]\caption{Time-dependent uncertainty (for python script see \cite{FGF})
    }\label{algorithmTime}
    \begin{algorithmic}
        \REQUIRE $M\in\mathds{N}$
        \FOR{each sample $j=0$ to $M-1$}
            \STATE Create $a^l$ defined in equations~\eqref{eq:OUAppr} and~\eqref{eq:OUApprDs} as approximation of $a(t)$
            \STATE Define the piecewise constant function $\hat a(t) = a^l$, for $t\in[l\Delta s,(l+1) \Delta s)$
            \STATE $t\leftarrow 0$, $n\leftarrow 0$
            \STATE Initialize cell averages $u_i^0$ for each cell $[x_{i-1/2},x_{i+1/2}]$
            \WHILE{$t< T$}
            \STATE Find $\Dt$ such that $\Phi(t+\Dt;t,\Dx) = |\int_{t}^{t+\Dt}\hat a(t) dt| - C_0\Dx = 0$ (and $t+\Dt\leq T$)
                \STATE Set $A^n = \int_{t}^{t+\Dt}\hat a(t) dt$
                \STATE With (periodic) boundary conditions for $u_i^n$ apply time step
                \STATE $u_i^{n+1} =u_i^n - 1/\Dx \left(  \text{max}\,(A^n,0)(u_i^n - u_{i-1}^n) 
                                    + \text{min}\,(A^n,0)(u_{i+1}^n - u_i^n)\right) $
                \STATE $t\leftarrow t + \Dt$
                \STATE $n\leftarrow n + 1$
            \ENDWHILE
        \ENDFOR
    \end{algorithmic}
\end{algorithm}
We would like to emphasize that the calculation of $\Dt$ in Algorithm~\ref{algorithmTime} is an important part of the algorithm.
We implemented bisection for root-finding in the following way. Given a time $t$ we first increase $l$ (from the previous value, initially 0) 
until we have $(l+1)\Delta s>t$ and $\Phi((l+1)\Delta s;t,\Dx) \geq 0$. Then we use bisection to find the root of $\Phi$ in the interval $[l\Delta s, (l+1)\Delta s]$.

This algorithm for finding $\Dt$ allows for a large stepsize where possible, while still approximating the time interval accurately. The advantage of a larger step size in the finite volume method is less numerical diffusion. In order to further reduce numerical diffusion of the base scheme, Algorithm~\ref{algorithmTime} can easily be extended to be second order in space and time as follows.

\subsection{Second order base scheme.}\label{sec:secondorder}
For the second order scheme in space and time two further ingredients are needed in each time step: using a non-oscillatory second order reconstruction by limiters (see~\cite{LEV1,VL1,HEOC1,SO1}) and the second order time stepping (see equation~\eqref{eq:SSP}).

To achieve second order accuracy in space it is standard (see, e.g.,~\cite{LEV1}) to replace the piecewise constant approximation $u_i$ of $u$ with a non-oscillatory piecewise linear reconstruction in-order to obtain second-order spatial accuracy. There are a variety of reconstructions including the popular TVD-MUSCL limiters (see, e.g.,~\cite{VL1}), ENO reconstruction (see, e.g.,~\cite{HEOC1}) and WENO reconstruction (see, e.g.,~\cite{SO1}). 
In this article we present results for the minmod and the superbee limiter, see for instance~\cite{LEV1}. We choose those two from a wide range of possible limiters, because both are TVD (total variation diminishing), but the minmod is the most "pessimistic" and the superbee is the most "optimistic" limiter in the TVD regime.

To present a scheme that is second-order in time, we use the second-order strong-stability preserving Runge--Kutta (SSP) time stepping given by
\begin{align}\label{eq:SSP}
    \begin{split}
        u^{\ast}_i &= u^n_i + \Dt  \boldsymbol{\mathcal{F}}^n_i, \\
        \quad u^{\ast \ast}_i &= u^{\ast}_i + \Dt
        \boldsymbol{\mathcal{F}}^\ast_i, \\
        u^{n+1}_i &= \frac{1}{2}(u^n_i + u^{\ast \ast}_i),
    \end{split}
\end{align}
where $\boldsymbol{\mathcal{F}}^n_i$ and $\boldsymbol{\mathcal{F}}^\ast_i$ are the numerical approximation of the fluxes, see e.g.~\cite{GST1}.
The time step is determined by a standard CFL condition. For both first and second order schemes we use a Courant number of $C_0 = 0.45$, see Algorithm~\ref{algorithmTime}. Although we have superconvergence, i.e. in some cases the upwind scheme reproduces the exact solution for $C_0 = 1$, we use a lower Courant number since "superconvergence" is not representative for typical schemes or more involved problems.

\subsection{Measurement of errors}
We are interested in measuring the error of the Monte Carlo estimator.

Let $\E(u)$ be the expectation of the exact solution $u$ and $u_{i,m}$ the numerical approximation to the solution of the partial differential equation for the $m$-th realization.
Then, the relative approximation error of the expectation in the $L^1$ norm is given by
\begin{equation}\label{eq:E_appr}
    \eps_{\text{appr}}(t) = \frac{\Delta x \sum_i | (\frac{1}{M} \sum_m u_{i,m}(x_i,t) ) - \E(u(x_i,t^n)) |} {\Delta x \sum_i | \E(u(x_i,t)) |}.
\end{equation}

\begin{remark}
    It is interesting to observe that the approximation error $\eps_{\text{appr}}$ is bounded by the sum of the numerical error $\eps_{\text{num}}$ of the base method and the pure Monte Carlo error $\eps_{\text{MCM}}$, that is
    \begin{equation}\label{eq:errorest}
        \eps_{\text{appr}}(t) \leq \eps_{\text{num}}(t) + \eps_{\text{MCM}}(t).
    \end{equation}
    Here, the relative $L^1$-error of the Monte Carlo approximation is given by
    \begin{equation}\label{eq:E_MCM}
        \eps_{\text{MCM}}(t) = \frac{\Delta x (\sum_i | \frac{1}{M} \sum_{m=1}^M u_{m}(x_i,t)  - \E(u(x_i,t)) |)} {\Delta x \sum_i| \E(u(x_i,t)) |},
    \end{equation}
    where $u_m$ denotes the exact solution of the partial differential equation for the $m$-th realization.
    The relative approximation error in the $L^1$-norm of the deterministic numerical method is
    \begin{equation}\label{eq:E_num}
        \eps_{\text{num}}(t) = \frac{\Delta x \sum_i | \frac{1}{M} \sum_m (u_{i,m}(x_i,t) - u_m(x_i,t) )|} {\Delta x \sum_i | \E(u(x_i,t)) |},
    \end{equation}
    Using the triangle inequality, it is trivial to show the relationship \eqref{eq:errorest}. If one uses the (squared) mean-squared errors (i.e. $L^2$-errors) then one may even show equality. 

    Relation~\eqref{eq:errorest} shows that the approximation error is bounded by the dominating part of the sum of the numerical error and the pure Monte Carlo error. The Monte Carlo method converges with the rate $1/2$ in the number of samples in mean square and is independent of the resolution of the grid, i.e. the size of $\Dx$. On the other hand, the numerical method, being first order, converges with $\mathcal{O}(\Dx)$ for each single realization, independent of the number of Monte Carlo samples.
    Therefore, equation~\eqref{eq:errorest} suggests that our Monte Carlo method is most efficient if $\eps_{\text{num}} \simeq \eps_{\text{MCM}}$
\end{remark}

Similarly,
let $\Var(u)$ be the variance of the exact solution $u$ and $u_{i,m}$ the numerical approximation to the solution of the partial differential equation for the $m$-th realization.
Then the absolute approximation error of the variance in the $L^1$ norm is given by
\begin{equation}\label{eq:V_appr}
    \delta_{\text{appr}}(t) = \Delta x \sum_i | (\frac{1}{M} \sum_m (u_{i,m}(x_i,t)  - \mu_{i,m}(x_i,t) )^2 ) - \var(u(x_i,t)) |,
\end{equation}
with $\mu_{i,m}$ denoting the (empirical) expectation of $u_{i,m}$.
%

\begin{figure}[htbp]
    \centering 
    \begin{tabular}{lr}
        \subfigure[Exact expectation and 3 sample solutions]{\includegraphics[width=0.49\linewidth,type=png,ext=.png,read=.png]
        {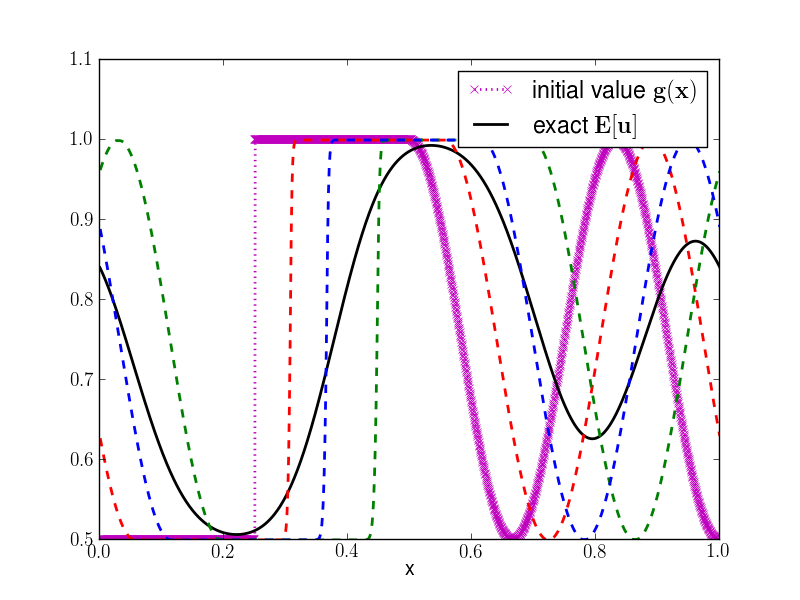}} 
        &
        \subfigure[Exact variance]{\includegraphics[width=0.49\linewidth,type=png,ext=.png,read=.png]
        {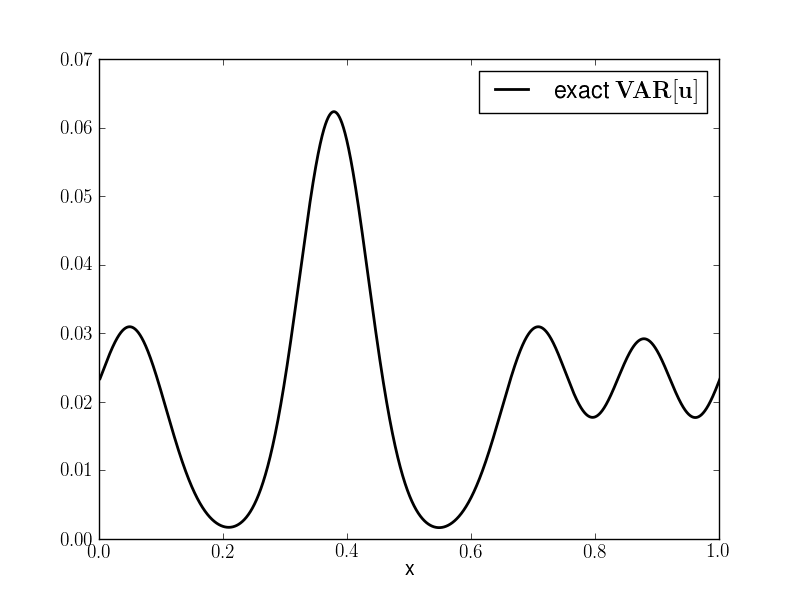}}
        \\
        \subfigure[Corresponding 3 approximations of the Ornstein--Uhlenbeck process]{\includegraphics[width=0.49\linewidth,type=png,ext=.png,read=.png]
        {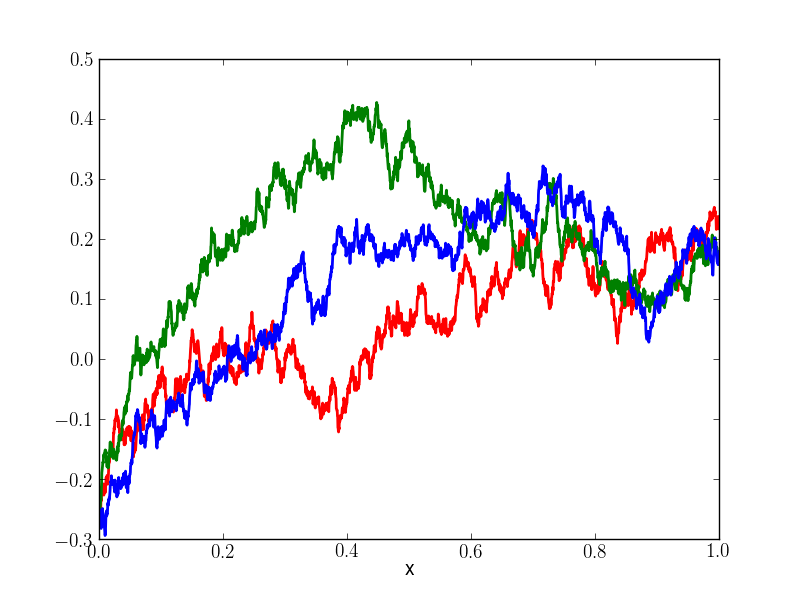}} 
        &
        \begin{parbox}{0.49\textwidth}{\hskip -4.5cm \vskip -5.5cm
                {\caption{The time-dependent problem given in Equation~\eqref{eq:problemTime} with $a(0)=-1/4$, for the parameter set $(\mu, \theta, \sigma) = (1/4, 4, 1/\sqrt{10})$. In (a) three sample solutions for the second order scheme with minmod limiter using 1600 mesh points and in (c) the corresponding approximations of the Ornstein--Uhlenbeck process are shown. We see the exact variance in (b) and the exact expectation in (a) at time $t=1$.}
            \protect \label{fig:problemTimeExact}
            }}
        \end{parbox}
    \end{tabular}
\end{figure}
\subsection{Simulation results of time-dependent uncertainty.}
In the following, we test the Monte Carlo method described in Algorithm~\ref{algorithmTime}.
In order to avoid numerical effects from boundary conditions we define the partial differential equation in expression~\eqref{eq:problemTime} on a spacial domain $[x_L,x_R]=[0,1]$ with periodic boundary conditions for both $u$ and the initial condition $g$, i.e.
\begin{equation}\label{eq:problembc}
    u(x_L,t,\omega) = u(x_R,t,\omega), \quad \text{ for all } t\geq0, \omega \in \Omega.
\end{equation}

In general initial conditions for hyperbolic problems consist of both smooth and discontinuous parts. In order to test our numerical schemes properly we therefore choose the initial condition to contain a sine wave and a jump-discontinuity, as shown in Figure~\ref{fig:problemTimeExact}(a).
We choose the deterministic initial condition for the Ornstein--Uhlenbeck process to be
$$a(0) = -\mu.$$
Three typical sample paths of Equation~\eqref{eq:OU} are plotted in Figure~\ref{fig:problemTimeExact}(c) for the parameter set $(\mu, \theta, \sigma) = (1/4, 4, 1/\sqrt{10})$ with $t\in[0,1]$.
As expected the Ornstein--Uhlenbeck process starts at $a_0=-\mu$ and (since $\theta>0$) fairly quickly relaxes to values around $+\mu$. Figure~\ref{fig:problemTimeExact}(a) shows the according three approximations to the (sample) solutions $u$ to the partial differential Equation~\eqref{eq:problemTime} for the different realizations of the Gaussian process $a$ shown in (c). They are obtained from Algorithm~\ref{algorithmTime} with a second order scheme using the minmod-limiter and with $1600$ mesh points. Since the samples $a$ start at a negative value, the initial profile $u$ gets advected to the left at first. But as time progresses, those sample paths eventually have positive values $a$ and therefore the solution $u$ of the PDE starts moving to the right again.

We can see in Figure~\ref{fig:problemTimeExact}(a) that the expectation $E(u)$ at time t consists of the initial function $g$ transported with speed $\hat{\mu}$ and smeared with the rate $\hat{\sigma}$, according to Theorem~\ref{theorem_lineartime}. The variance, shown in (b), is highest at the transported initial (now smoothed out) jump discontinuity. 
\begin{figure}[htbp]
    \centering 
    \begin{tabular}{lr}
        \subfigure[Error expectation]{\includegraphics[width=0.49\linewidth,type=png,ext=.png,read=.png]{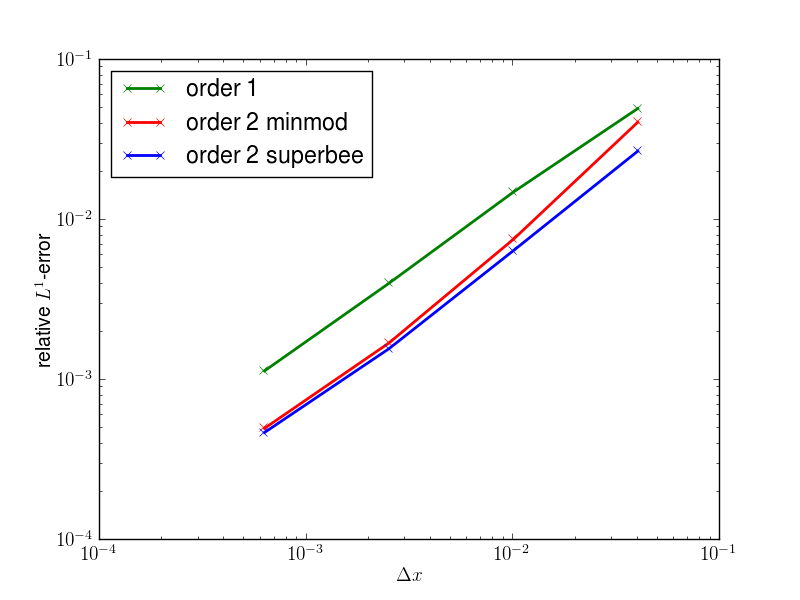}}
        &
        \subfigure[Error variance]{\includegraphics[width=0.49\linewidth,type=png,ext=.png,read=.png]{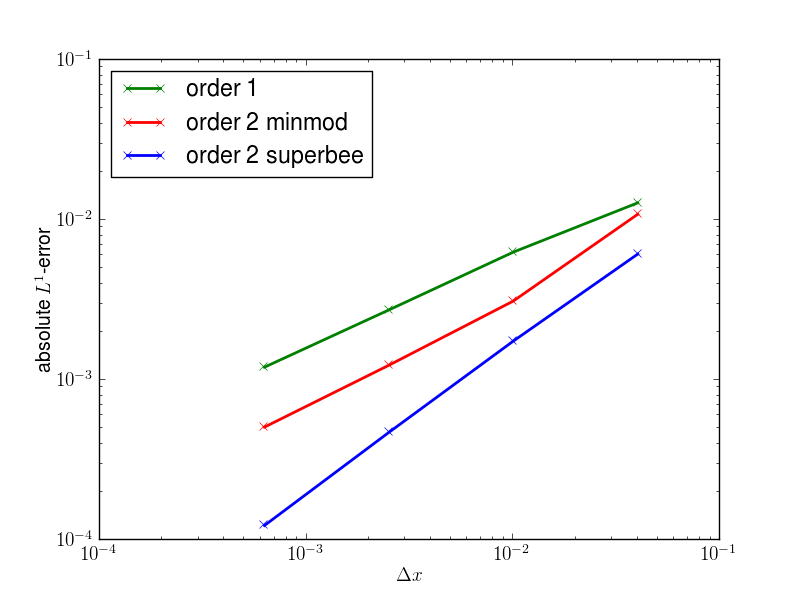}}
        \\
    \end{tabular}
    \caption{The $L^1$-errors for the time-dependent problem (see Equation~\eqref{eq:problemTime}). Dependence of the errors defined in Equations~\eqref{eq:E_appr} and \eqref{eq:V_appr} on the number of grid points at time $t=1$ with $a(0)=-1/4$ and for the parameter set $(\mu, \theta, \sigma) = (1/4, 4, 1/\sqrt{10})$ using $M=10^6$ Monte Carlo simulations.}
    \protect 
    \label{fig:problemTimeConvergence}
\end{figure}

Next, we test the convergence of the schemes described in Algorithm~\ref{algorithmTime} with respect to mesh refinement. Therefore, we choose a high number of samples $M$ in the Monte Carlo simulation, such that the dominating error of $\eps_{\text{appr}}$ is the one of the numerical base method, see Inequality~\eqref{eq:errorest}.
We compare first and second order base schemes with a Courant number of $0.45$.
We present plots for the approximation errors of the first two moments $\eps_{appr}$ and $\delta_{appr}$. As expected, Figure~\ref{fig:problemTimeConvergence} shows that overall the second order schemes have a smaller error than the first order scheme. Among the two second order schemes, the one using the superbee limiter has the smaller error, especially the error for the variance.

The results indicate that the approximation of the distribution of the solution to Equation~\eqref{eq:problemTime} given by the Monte Carlo method presented in Algorithm~\ref{algorithmTime} converges to the exact solution as $M\rightarrow\infty$ and $\Delta x \rightarrow 0$ simultaneously. This concludes the time dependent case and we continue with space dependent coefficients.

\section{Space dependent uncertainty.}
In this section we investigate the case where the uncertainty depends on the space variable, i.e. the advection parameter is a random field with a given covariance.
More specifically, we look at the following equation
\begin{align}\label{eq:problemSpace}
    \begin{split}
        u_t(x,t,\omega) + a(x,\omega)\,u_x(x,t,\omega) &= 0,\\
        u(x,0,\omega) &= g(x).
    \end{split}
\end{align}
The coefficient $a$ is then modeled as a random field, which takes values in a function space $H$ over the domain $D\subset\R$, here $H:=L^2(D)$. We assume that the random field $a$ is characterized by its mean and its covariance operator. More precisely, we assume that there exists a covariance operator $Q\in L_1^+(H)$, where $L_1^+(H)$ denotes the space of all nonnegative, symmetric and nuclear operators in $H$. For every such operator the Hilbert--Schmidt theorem on the spectral representation holds: there exists an orthonormal basis $(e_i,i\in\N)$ of $H$ such that $Qe_i=\lambda_i e_i$, where all $\lambda_i\geq 0$ in the sequence $(\lambda_i,i\in\N)$ and 0 is its only accumulation point. Such a random field is characterized by its Karhunen--Lo\`eve expansion
\begin{equation}\label{eq:KL_expansion}
    a(x,\omega) = \mu + \sum_{i\in\N} \sqrt{\lambda_i}\beta_i(\omega)e_i(x).
\end{equation}
Here, $(\beta_i,i\in\N)$ is a sequence of independent normally distributed random variables and $\mu$ is finite. A similar expression holds if $a$ is a L\'evy field. Then we have
\begin{equation}\label{eq:KL_expansion_Levy}
    a(x,\omega) = \mu  + \sum_{i\in\N} \sqrt{\lambda_i} L_i(\omega)e_i(x).
\end{equation}
In this case, $(L_i,i\in\N)$ is a sequence of real-valued, orthogonal Poisson-distributed random variables.

\begin{remark}
    We would like to point out that Equation~\eqref{eq:problemSpace} is not a conservative equation. Solutions of the conservative advection equation with space-dependent variables, i.e. $u_t + (a(x)u)_x = 0$ consist in general of delta functions (see~\cite{LEV1} Chapter 16.4).
\end{remark}

It is challenging to derive a closed form for the distribution of the solutions to Equation~\eqref{eq:problemSpace}. We would like, however, to present a possible way in that direction, by showing a bound on the characteristic curves of that equation.

\subsection{Theoretical results.}
We could not find any hint in the literature to a closed form solution to Equation~\eqref{eq:problemSpace}. However, one could find the distribution of the solution by looking at the characteristic curves. The characteristic curves are the solutions of the autonomous ordinary (in this case stochastic) differential equation
\begin{equation}\label{eq:char_curve_ode}
    \frac{dX(t)}{dt} = a(X(t),\omega) \qquad \Leftrightarrow \qquad dX(t) = a(X(t),\omega)\,dt.
\end{equation}
Equation~\eqref{eq:problemSpace} is linear and therefore the solution along the characteristic curves is constant. Furthermore, for a linear advection equation, even one with variable coefficients, the characteristics will never cross, see~\cite[p. 208]{LEV1}.
Using the Karhunen-Lo\`eve expansion~\eqref{eq:KL_expansion} one can write the equation of the characteristics~\eqref{eq:char_curve_ode} as
\begin{equation}\label{eq:SDE}
    X(t) 
    = X(0) + \mu t + \int_0^t \sum_{i\in\N} \beta_i(\omega)f_i(X(s))\,ds,
\end{equation}
where we set $f_i(\cdot)=\sqrt{\lambda_i}e_i(\cdot)$ for all $i\in\N$. 
We want to proof the existence of a solution to Equation~\eqref{eq:char_curve_ode} in the space $L^2(\Omega;C([0,T];\R))$ of square integrable functions with values in $C([0,T];\R)$. One important example of a Covariance operator is given by the Gaussian covariance kernel. For instance in the overview article~\cite{F11} one can find expressions for the eigenvalues $(\lambda_i,i\in\N)$ and eigenfunctions $(e_i,i\in\N)$ of the Gaussian covariance operator with integral kernel $q(x,y) = \e^{\frac{(x-y)^2}{2}}$. They are given by
\begin{equation*}
    \lambda_i = \frac{1}{(1+\sqrt{3}/2)^{1/2}}\frac{1}{(2+\sqrt{3})^i}
\end{equation*}
and
\begin{equation*}
    e_i(x) = \frac{3^{1/8}}{\sqrt{2^i i!}} \e^{-(\sqrt{3}-1)\frac{x^2}{2}} H_i(3^{1/4} x),
\end{equation*}
where $H_i$ denotes the $i$-th Hermite polynomial. Then, for each $i\in\N$, $f_i$ is bounded since
\begin{equation*}
    |f_i(X(s))| \leq \frac{3^{1/8}}{(1+\sqrt{3}/2)^{1/4}} \left(\frac{3}{(2+\sqrt{3})}\right)^{i/2} \frac{1}{\sqrt{2^ii!}} \e^{-\frac{(\sqrt{3}-1)}{2}X(s)^2} |X(s)^i|.
\end{equation*}
And further 
\begin{equation*}
    \frac{1}{\sqrt{2^ii!}}\e^{-\frac{(\sqrt{3}-1)}{2}X(s)^2} |X(s)^i|\leq 1
\end{equation*}
since for any $y\in\R$ we have
\begin{equation*}
    y \leq \e^{(b/i) y^2}\e^{1/2(\ln(2)+\ln(i!)/i)},
\end{equation*}
where $b=\frac{(\sqrt{3}-1)}{2}$. Overall it follows that, for all $s\in[0,T]$, $|f_i(X(s))|\leq C<+\infty$ and, therefore, $|e_i(X(s))|<C$. This result can be generalized for all $Q\in L^+_1(H)$. With this in hand we show that the solution to the stochastic differential Equation~\eqref{eq:SDE} $X\in L^2(\Omega;C([0,T];\R))$.

\begin{lemma}
 If $\sum_{i\in\N}\sqrt{\lambda_i}<+\infty$ and $X(0)\in L^2(\Omega;\R)$, then the solution to Equation~\eqref{eq:SDE} $X$ belongs to $L^2(\Omega;C([0,T];\R))$.
\end{lemma}

\begin{proof}
We have by the definition of the norm of $L^2(\Omega;C([0,T];\R))$
\begin{align*}
 \|X\|_{L^2(\Omega;C([0,T];\R))}^2 &:= \E(\sup_{t\in[0,T]}|X(t)|^2)\\
 &\leq C \big(\E|X(0)|^2 + \E(\sup_{t\in[0,T]}|\int_0^t a(X(s))\,ds|^2)\big)\\
 &\leq C \big(\E|X(0)|^2 + \mu^2 T + \E(\sup_{t\in[0,T]}\int_0^t |\sum_{i\in\N} \sqrt{\lambda_i}\beta_i e_i(X(s))|^2\,ds)\big).
\end{align*}
The last term is further bounded by
\begin{align*}
 \E(\sup_{t\in[0,T]}\int_0^t &|\sum_{i\in\N} \sqrt{\lambda_i}\beta_i e_i(X(s))|^2\,ds)\\
 &\leq \E(\sup_{t\in[0,T]}\int_0^t \sum_{i\in\N} \sqrt{\lambda_i}\beta_i^2 \sum_{i\in\N} \sqrt{\lambda_i}|e_i(X(s))|^2\,ds)\\
 &\leq \E(\sum_{i\in\N} \sqrt{\lambda_i}\beta_i^2 \sum_{i\in\N} \sqrt{\lambda_i}\sup_{t\in[0,T]}\int_0^t |e_i(X(s))|^2\,ds)\\
 &\leq \sum_{i\in\N} \sqrt{\lambda_i} \E(\beta_i^2) \sum_{i\in\N} \sqrt{\lambda_i}C(T)\\
 &\leq C(T) (\sum_{i\in\N} \sqrt{\lambda_i})^2,
\end{align*}
where we used the Cauchy--Schwarz inequality, the bound on the eigenfunctions and that the inpendent random variables $\beta_i$ are standard normally distributed, for $i\in\N$. So overall we have the bound
\begin{equation*}
 \|X\|_{L^2(\Omega;C([0,T];\R))}^2 \leq C(T) \big(\E|X(0)|^2 + \mu^2 + (\sum_{i\in\N} \sqrt{\lambda_i})^2\big).
\end{equation*}

\end{proof}

However, this is not a constructive approach to a solution, albeit it justifies the use of a Monte Carlo method. Since we are not aware of any results on closed form solutions we consider numerical approximations in the next section. 
The (additional) assumption that the sequence $(\sqrt{\lambda_i},i\in\N)$ is summable is for many common covariance kernels fulfilled. In particular, the example of the Gaussian covariance kernel has exponentially decaying eigenvalues. We remark further that this also holds for a L\'evy random field as defined in Equation~\eqref{eq:KL_expansion_Levy}. 

\subsection{Discretizations of space-dependent uncertainty.}
As in the time-dependent case, we employ a Monte Carlo based method for the approximation of the (moments of the) solution to Equation~\eqref{eq:problemSpace}. Using the same notation as in Section~\ref{sec:firstorderTime}, we start by describing a first order base scheme for each realization of the random field in Equation~\eqref{eq:problemSpace}.
Again, to avoid numerical artifacts from the boundary, we use periodic boundary conditions for the random field
\begin{equation}\label{eq:problembcSpace}
    a(x_L,\omega) = a(x_R,\omega), \quad \omega \in \Omega.
\end{equation}
and the functions $u$ and $g$, see Equation~\eqref{eq:problembc}.
\begin{figure}[htbp]
    \includegraphics[width=1.\linewidth,type=png,ext=.png,read=.png]{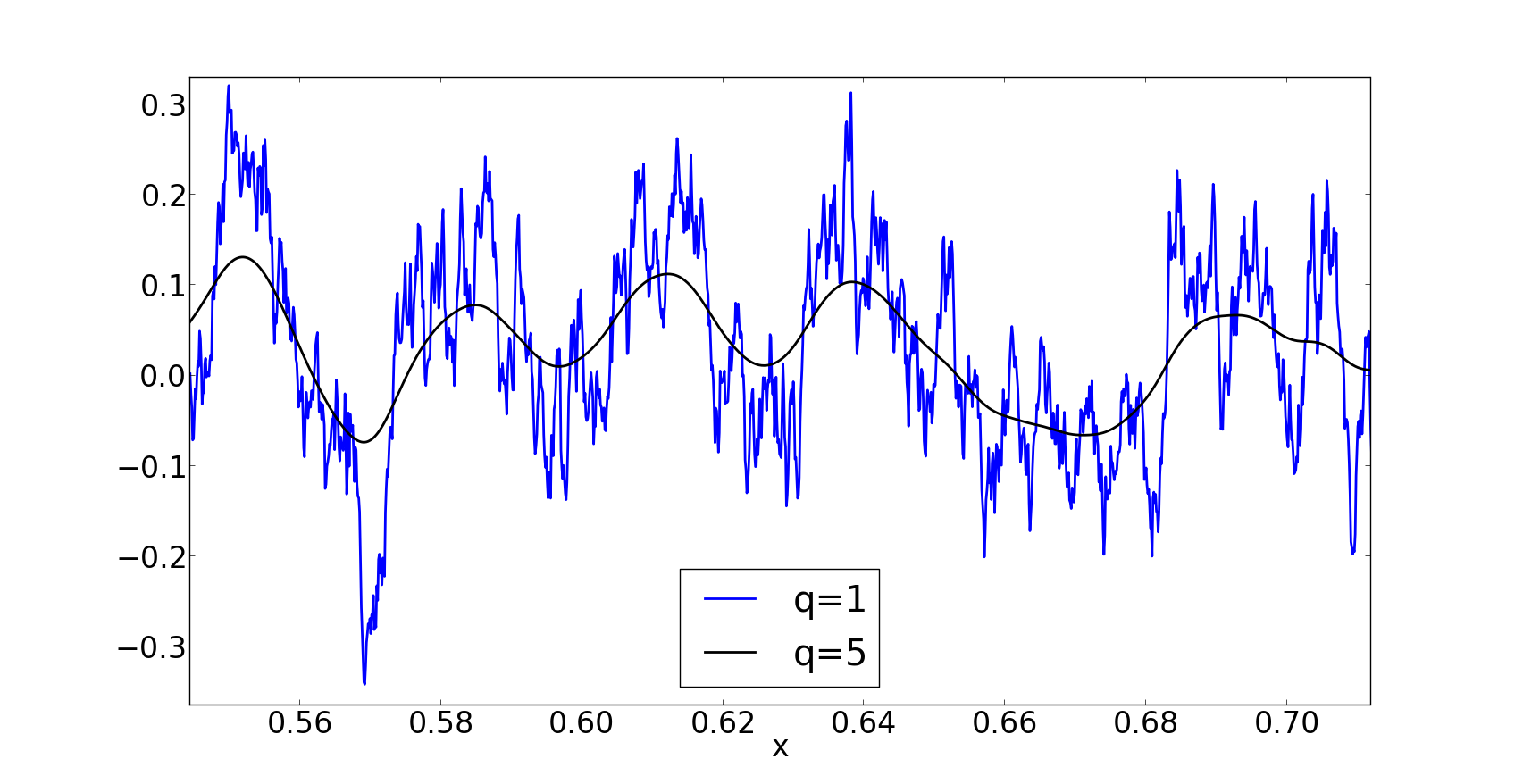}
    \caption{Comparison of correlated random fields for $q=1$ and $q=5$, generated using algorithm~\ref{algorithmSpace} with 8192 points in $[0,1]$. The larger $q$ the stronger the correlation and therefore the less oscillatory the random field.}
    \protect \label{fig:comparison}
\end{figure}
Using periodicity we define the Gaussian random field $a$ in the following manner. Let $W$ be a Gaussian white noise random field on $\R$ and $\gamma(\xi)$, for $\gamma:\R\rightarrow\R_+$ an even and positive function. Then, we set for any $\mu,\sigma \in \R$ with $\sigma\geq 0$,
\begin{equation}\label{eq:correlatedRF}
    a(x) = \mu + \sqrt{\sigma}(\mathcal{F}^{-1}\sqrt{\gamma}\mathcal{F}W)(x),
\end{equation}
where $\mathcal{F}$ denotes the Fourier transform and $\mathcal{F}^{-1}$ its inverse. Then, since $W$ is centered Gaussian, so is $a-\mu$ and the covariance of $a-\mu$ is given by
\begin{equation}
    \E\big((a(x)-\mu)(a(y)-\mu)\big) = \int_\R e^{-2\pi i p(x-y)}\sigma\gamma(p) dp,\quad x,y\in\R,
\end{equation}
where $i$ is the imaginary unit.
This approach leads to a fast simulation of Gaussian random fields. A typical family of functions for the Lebesgue density $\gamma$ is given by
\begin{equation}\label{eq:correlation}
    \gamma(\xi) = (1 + \xi^2)^{-q}, \quad q\in\N,\, q\geq1.
\end{equation}
The larger the parameter $q$ the higher the spacial correlation of the Gaussian random field $a$, see Figure~\ref{fig:comparison}.
In order to approximate the solution to Equation~\eqref{eq:problemSpace} we propose the following Monte Carlo based approach. It uses a fast approximation of the Gaussian random field $a$ as provided in~\cite{LP}. For each realization of the random field the discretization of Equation~\eqref{eq:problemSpace} is standard, see for instance \cite[Chapter 9]{LEV1}. For a first order scheme we introduce Algorithm~\ref{algorithmSpace}.
\begin{algorithm}[H]\caption{Space-dependent uncertainty (for python script see \cite{FGF})}\label{algorithmSpace}
    \begin{algorithmic}
        \REQUIRE $M\in\mathds{N}$
        \FOR{each sample $j=0$ to $M-1$}
            \STATE Set $\Omega=50$
            \STATE Set $\xi_{i-1/2} = \begin{cases}
                (i-1)/\Omega,& \text{if } i\leq I/2+1\\
                (I-(i-1))/\Omega,& \text{else.}
                \end{cases}
                $
            \STATE Set $\gamma_{i-1/2} = (1+\xi_{i-1/2}^2)^{-q}/\Omega$, for $i=1,...,I$
            \STATE Calculate $a_{1/2,...,I-1/2} = \mu + \mathcal{F}^{-1}\left(\sqrt{\gamma_{1/2,...,I-1/2}}\,\mathcal{F}\left(Z_{1/2,...,I-1/2}\right)\right),$ where $Z_{i-1/2} = \sqrt{\sigma/\delta} Y_{i-1/2}, \delta = \Omega/I$, with $Y_{i-1/2} \sim \mathcal{N}(0,1)$
            \STATE Use periodic boundary conditions $a_{I+1/2} = a_{1/2}$
            \STATE $t\leftarrow 0$, $n\leftarrow 0$
            \STATE Initialize cell averages $u_i^0$ for each cell $[x_{i-1/2},x_{i+1/2}]$
            \STATE Set $\Dt = \Delta t = c \Dx/\text{max}_i(|a_{i+1/2}|)$, where $c$ is the Courant number
            \WHILE{$t< T$}
                \IF{$t+\Delta t>T$}
                    \STATE $\Delta t = T-t$ 
                \ENDIF
                \STATE With (periodic) boundary conditions for $u_i^n$ apply time step
                \STATE $u_i^{n+1} =u_i^n - \Dt/\Dx \left(  \text{max}(a_{i-1/2},0)(u_i^n - u_{i-1}^n) 
                    + \text{min}(a_{i+1/2},0)(u_{i+1}^n - u_i^n)\right) $
                \STATE $t\leftarrow t + \Dt$
                \STATE $n\leftarrow n + 1$
            \ENDWHILE
        \ENDFOR
    \end{algorithmic}
\end{algorithm}
As before, the second order (in space and time) accurate scheme requires two further ingredients in each time step: using a non-oscillatory second order reconstruction using limiters and the second order time stepping (see Section~\ref{sec:secondorder}).

\subsection{Simulation results of space-dependent uncertainty.}
Figure~\ref{fig:comparison} shows two realizations of the Gaussian random fields generated by Algorithm~\ref{algorithmSpace}. As expected the random field is less oscillatory for $q=5$ compared to $q=1$, since the correlation of the random field is much stronger.

We start by pointing out that the variance of $a$ in Equation~\eqref{eq:correlatedRF} is independent of $x$. Thus, it makes sense in the simulations to choose $\mu$ to be $\zeta$ standard deviations of $a-\mu$, that is
\begin{equation}
    \mu \approx \zeta \sqrt{\Var[\sqrt{\sigma}(\mathcal{F}^{-1}\sqrt{\gamma}\mathcal{F}W)(x)]}.
\protect \label{eq:spacemu}
\end{equation}
Since the generated Gaussian random field $a$ is normally distributed this means that the probability that $a<0$ is $\frac{(1-\text{erf}(\zeta/\sqrt{2}))}{2}$.

\begin{figure}[htbp]
\centering 
\begin{tabular}{lr}
\subfigure[Expectation, $q=1$]{                            
\includegraphics[width=0.49\linewidth,type=png,ext=.png,read=.png]{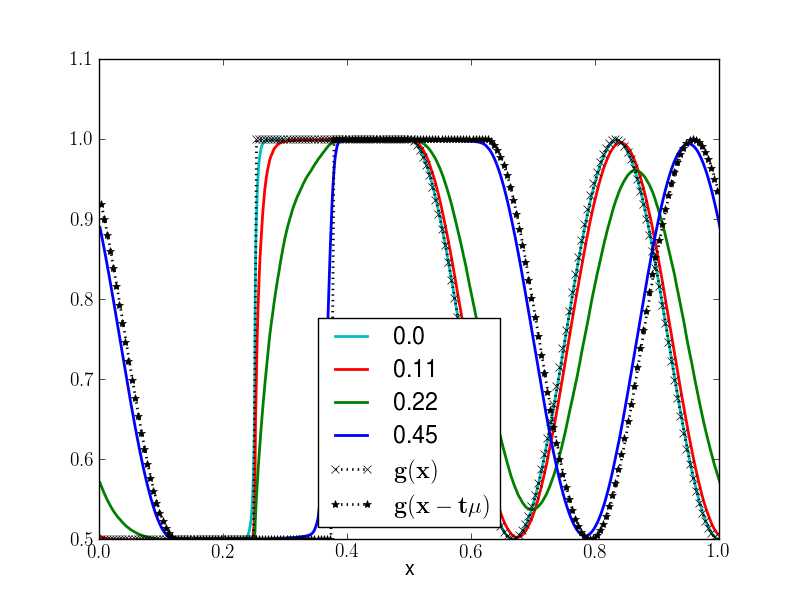}
}
& 
\subfigure[Variance, $q=1$]{
\includegraphics[width=0.49\linewidth,type=png,ext=.png,read=.png]{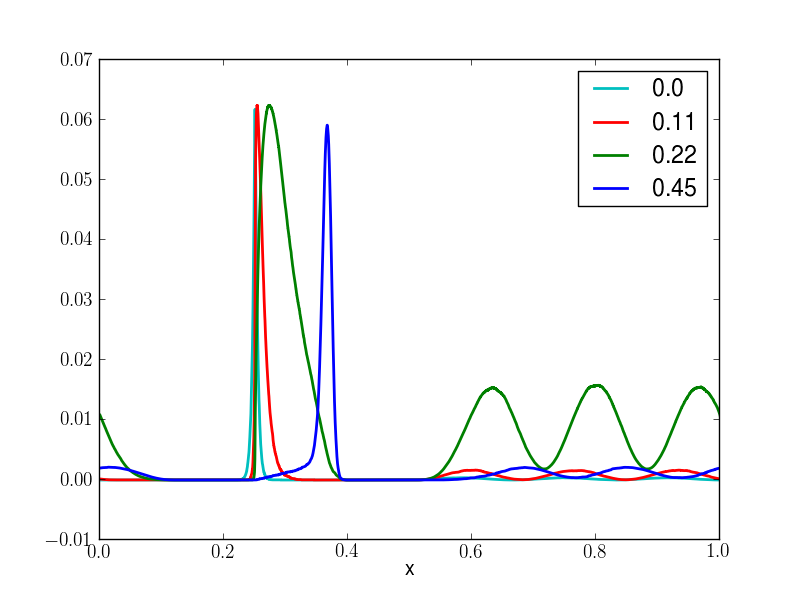}
}
\\                                                                                  
\subfigure[Expectation, $q=5$]{                 
\includegraphics[width=0.49\linewidth,type=png,ext=.png,read=.png]{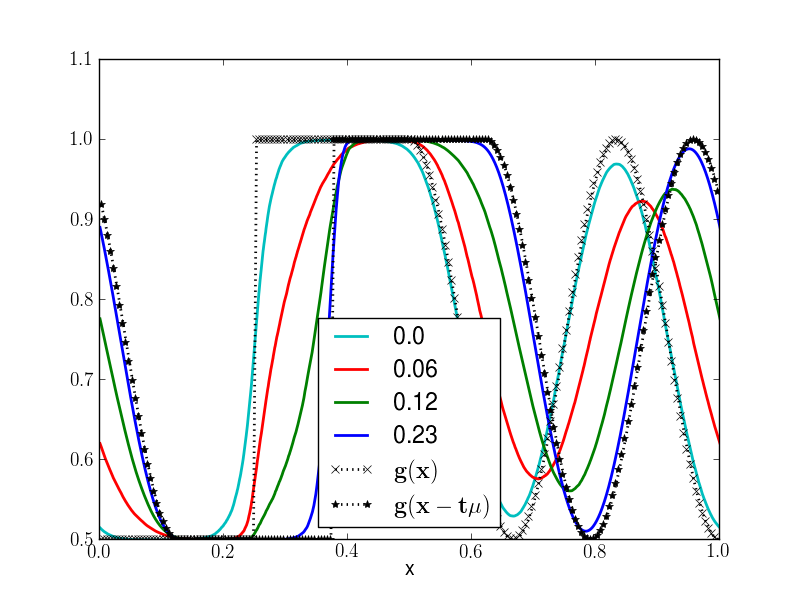}
}
&
\subfigure[Variance, $q=5$]{                            
\includegraphics[width=0.49\linewidth,type=png,ext=.png,read=.png]{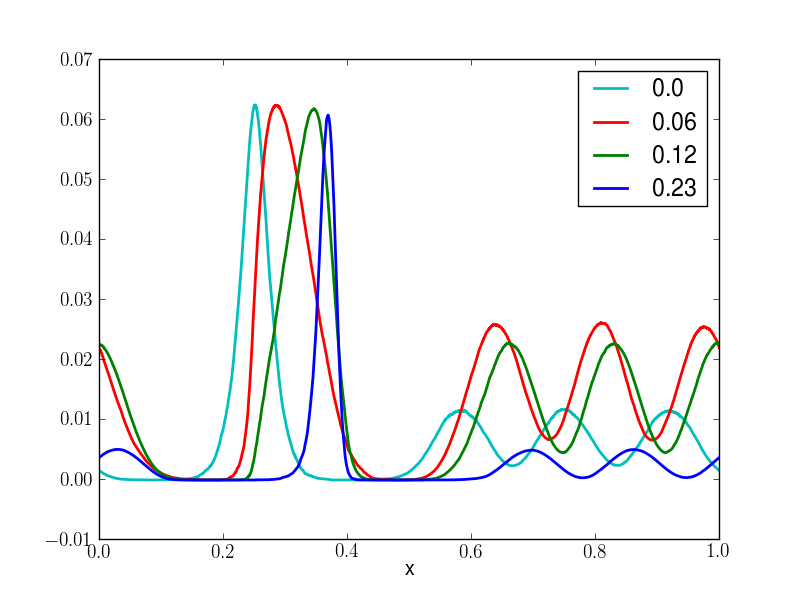}
}
\end{tabular}
\caption{Results using the second order minmod scheme for $\mu=0, \sigma=10, q=5$. If not noted otherwise, the expectation and variance are calculated with $10^4$ Monte Carlo samples and $2^{15}=32768$ mesh points.}
\protect \label{fig:dependencezeta}
\end{figure}

Figure~\ref{fig:dependencezeta} presents examples for $\zeta=0,1,2,4$, leading to the probability of $\approx50\%$, $16\%$, $2.3\%$ and $0.003\%$ negative values, respectively. This implies that the possibility for zero-crossings of $a$ varies with $\zeta$. Speaking in terms of the characteristic curves (see Equation~\eqref{eq:char_curve_ode}), such points will "trap" the solution at that point, and reduce the average propagation speed.

The solutions shown in Figure~\ref{fig:dependencezeta} were obtained using the second order minmod-based scheme described in Algorithm~\ref{algorithmSpace}, using $2^{15}=32768$ grid cells and $10^4$ Monte Carlo samples.
In order to be able to compare the dependence of $\zeta$ on the solution, we compute the solution up to time $t=c/\mu$ (for $\mu>0$) such that $x-t\mu=x-c$ is independent of $\mu$. For $\mu=0$ we choose $t=2$.

As can be seen in Figure~\ref{fig:dependencezeta} the expectation of the solution to Equation~\eqref{eq:problemSpace} depends heavily on $\zeta$. The larger $\zeta$ the more unlikely we get a zero-crossing of $a$ and therefore the average propagation speed is closer to the deterministic case. As expected, this effect is more pronounced for the less correlated Gaussian random field with $q=1$. For $\zeta=1$ the average propagation speed is almost reduced to zero.

In the extreme case when $\mu=0$, our numerical simulations suggest that the expectation of the solution $\E(u)$ is obtained by a convolution of the initial function $g$ with a Gaussian function. Figure~\eqref{fig:spacemu0q5}~(a) presents $g$ and $\E(u)$ along with two sample solutions, obtained by a second order minmod-based scheme described in Algorithm~\ref{algorithmSpace}, using $2^{13}=8192$ grid cells and $10^5$ Monte Carlo samples. In Figure~\ref{fig:spacemu0q5}(b) we can see that the bulk of the variance is located around the initial discontinuity. Unfortunately, even in this simple case we were not able to derive a closed-form solution. Based on our experiments we claim that the parameters of the aforementioned Gaussian function depend intricately on the first and second moment of the Gaussian random field $a$.

\begin{figure}[htbp]
\centering 
\begin{tabular}{lr}
\subfigure[Expectation and 2 sample solutions]{                            
\includegraphics[width=0.49\linewidth,type=png,ext=.png,read=.png]{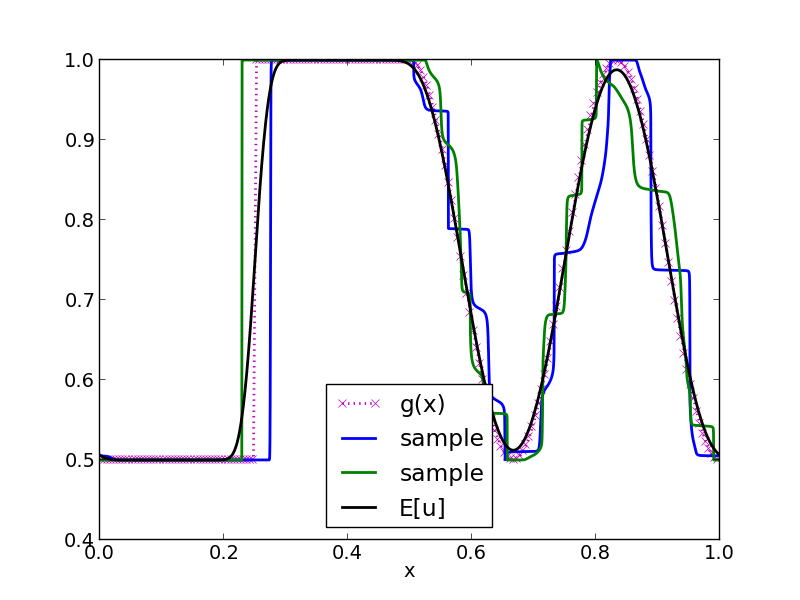}
} &
\subfigure[Variance]{
\includegraphics[width=0.49\linewidth,type=png,ext=.png,read=.png]{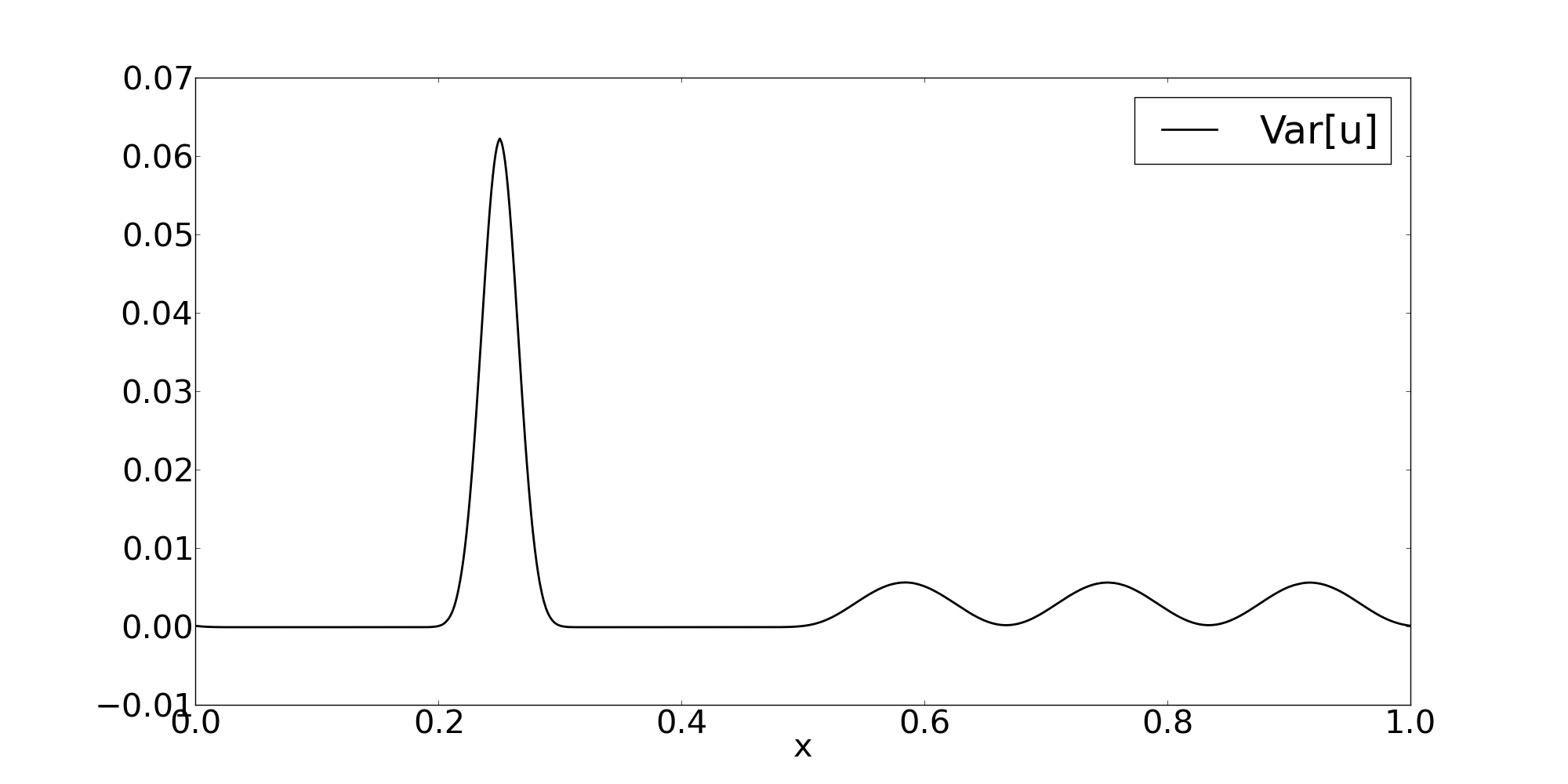}
}
\end{tabular}
\caption{Results using the second order minmod scheme for $\mu=0, \sigma=10, q=5$. If not noted otherwise, the expectation and variance are calculated with $10^5$ Monte Carlo samples and $2^{13}=8192$ mesh points.}
\protect \label{fig:spacemu0q5}
\end{figure}
Finally, we present a convergence study for the second order minmod scheme described in Algorithm~\ref{algorithmSpace}. Figure~\ref{fig:spaceconvergence} shows the first two moments for $\zeta=2$. As we quadruple the number of points several times starting with $1024$ points we can see that both expectation and variance of the solution converge. Compared to $q=5$ we need more points for $q=1$ in order to have a good approximation to the underlying random field, since it is less correlated. But even in this case the moments of the solution converge. We would like to comment, that the simulation with $65536$ points with $10000$ Monte Carlo samples took two weeks running simultaneously on $10$ cores, so it would have taken roughly $5$ months on a single CPU. 
\begin{figure}[htbp]
\centering 
\begin{tabular}{lr}
\subfigure[Expectation, $q=5$]{                            
\includegraphics[width=0.49\linewidth,type=png,ext=.png,read=.png]{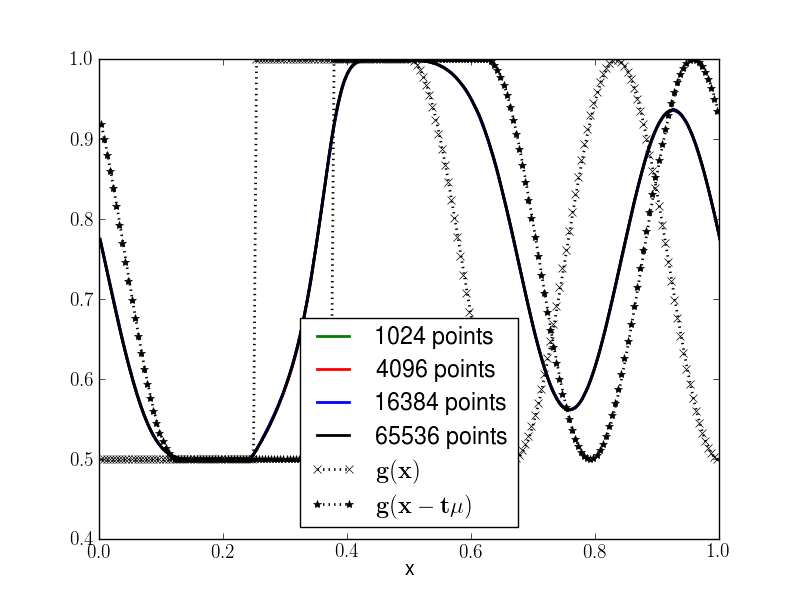}
}&
\subfigure[Variance, $q=5$]{
\includegraphics[width=0.49\linewidth,type=png,ext=.png,read=.png]{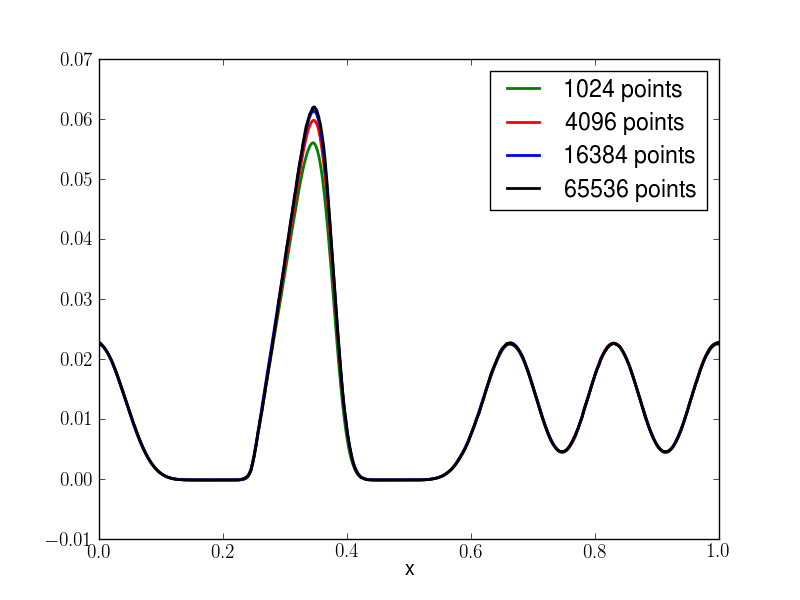}
}
\\
\subfigure[Expectation, $q=1$]{
\includegraphics[width=0.49\linewidth,type=png,ext=.png,read=.png]{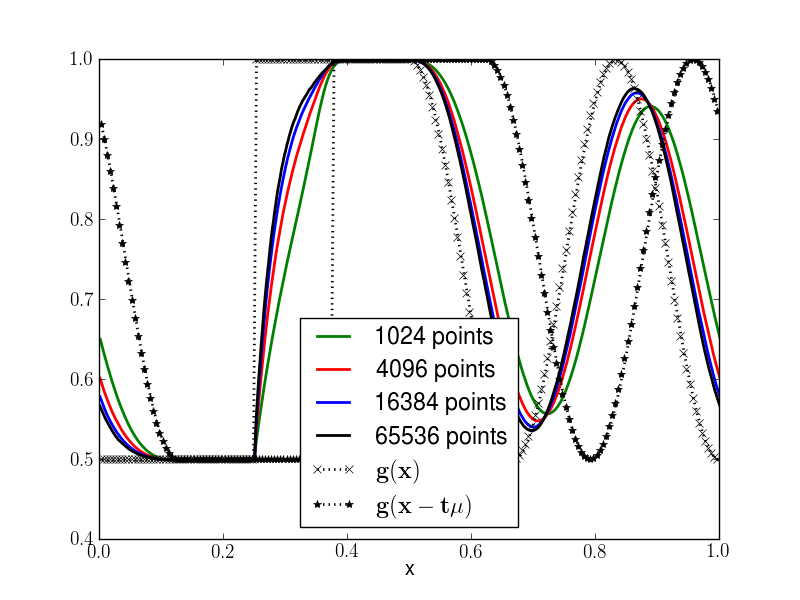}
}&
\subfigure[Variance, $q=1$]{
\includegraphics[width=0.49\linewidth,type=png,ext=.png,read=.png]{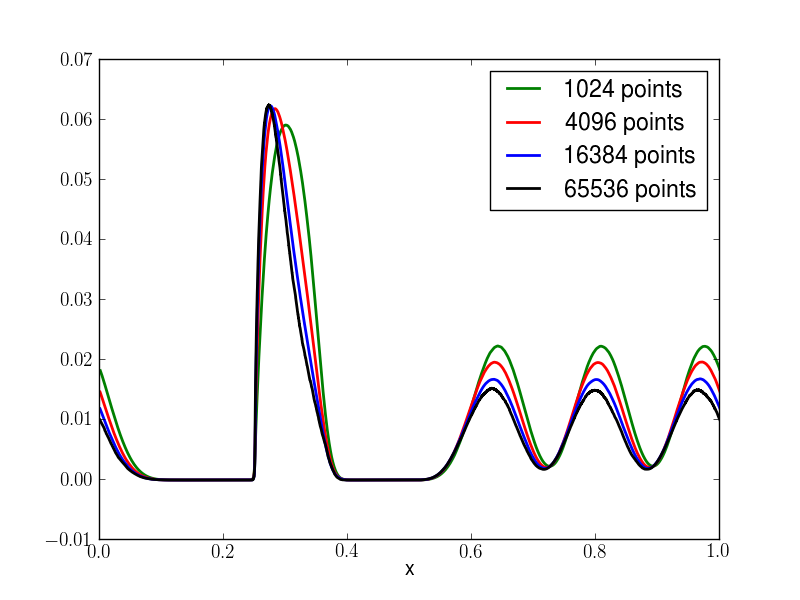}
}
\end{tabular}
\caption{Self-convergence study for second order minmod scheme with $M=10^4$ Monte Carlo samples for equation~\eqref{eq:spacemu} with $\zeta=2$.}
\protect \label{fig:spaceconvergence}
\end{figure}

As pointed out, there seem to be no closed-form solutions for the space-dependent uncertainty case. However, our numerical simulations indicate that uncertainty has a diffusive effect on $\E[u]$, similar to Problem~\ref{eq:problemTime}.
Furthermore, simulations suggest that the average propagation speed is affected by the stochastic term, that is the advection speed differs from the mean $\mu$ of the random field, given in Equation~\eqref{eq:correlatedRF}.

\section{Conclusion}
We have investigated numerical schemes for the approximation of the first and second moment of the solution of a hyperbolic problem with stochastic coefficients. We investigated the cases where the coefficient is given by a Gaussian process and a Gaussian/L\'evy random field. We introduced an adaptive scheme for the time-dependent problem which takes into account the special features of the Ornstein--Uhlenbeck process. Further, we gave closed form solutions for the (moments of the) distribution of the solution in the time-dependent case. We investigated the characteristic curves of the space-dependent problem where the stochastic coefficient is modeled by a Gaussian or L\'evy random field. We showed that the characteristic curves have finite variance. In the simulations, we put emphasize on the dependency of the correlation, mean and variance. 

We presented Monte Carlo based approximations for the distribution of the solutions to the stochastic partial differential equations for both the time- and the space-dependent case. We presented error plots showing convergence when applicable or showed self-convergence.
Naturally, the Monte Carlo approach could be extended to computationally advantageous multilevel methods, see for instance~\cite{BL12, MSS}.
For the space-dependent case, the numerical experiments suggest that the average speed of propagation intricately depends on the underlying Gaussian random field.

Finally, since we believe in reproducible science, the python scripts used to create the results in this paper are available at~\cite{FGF}.

\section*{Acknowledgement}
The authors would like to express their gratitude towards the University of Oslo, particularly the Center of Mathematics for Applications (CMA), the Eidgen\"ossische Technische Hochschule Z\"urich (ETH) and SINTEF ICT Oslo.

\bibliographystyle{siam}
\bibliography{uncertainty_barth_fuchs}
\end{document}